\documentclass[acmsmall]{acmart}

\def\BibTeX{{\rm B\kern-.05em{\sc i\kern-.025em b}\kern-.08emT\kern-.1667em\lower.7ex\hbox{E}\kern-.125emX}}

\setcopyright{acmcopyright}
\acmJournal{TOMS}
\acmYear{2019}\acmVolume{0}\acmNumber{0}\acmArticle{0}\acmMonth{0}
\acmDOI{0}

\usepackage{graphicx}
\usepackage{tikz}

\usepackage{xcolor}
\usepackage[linesnumbered,ruled,vlined]{algorithm2e}

\usepackage{multirow}

\SetCommentSty{mycommfont}

\SetKwInput{KwInput}{Input}                % Set the Input
\SetKwInput{KwOutput}{Output}              % set the Output

\begin{document}

%
% The ''title'' command has an optional parameter, allowing the author to define a ''short title'' to be used in page headers.
\title[Efficient solvers for Armijo's back-tracking problem]{Efficient solvers for Armijo's back-tracking problem}

%
% The ''author'' command and its associated commands are used to define the authors and their affiliations.
% Of note is the shared affiliation of the first two authors, and the ''authornote'' and ''authornotemark'' commands
% used to denote shared contribution to the research.
\author{I. F. D. Oliveira}
\email{ivodavid@gmail.com}
\affiliation{%
  \institution{Institute of Science, Engineering and Technology, Federal University of the Valleys of Jequitinhonha and Mucuri}
  \streetaddress{R. Cruzeiro, 1}
  \city{Te\'ofilo Otoni}
  \state{Minas Gerais}
  \country{Brazil}
  \postcode{39803-371}
}
\author{R. H. C. Takahashi}
\email{taka@mat.ufmg.br}
\affiliation{%
  \institution{Department of Mathematics, Federal University of Minas Gerais}
  \streetaddress{Av. Pres. Ant\^onio Carlos, 6627}
  \city{Belo Horizonte}
  \state{Minas Gerais}
  \country{Brazil}
  \postcode{31270-901}
}

%
% By default, the full list of authors will be used in the page headers. Often, this list is too long, and will overlap
% other information printed in the page headers. This command allows the author to define a more concise list
% of authors' names for this purpose.
\renewcommand{\shortauthors}{Oliveira and Takahashi}

%
% The abstract is a short summary of the work to be presented in the article.
\begin{abstract}
Backtracking is an inexact line search procedure that selects the first value in a sequence $x_0, x_0\beta, x_0\beta^2...$ that satisfies $g(x)\leq 0$ on $\mathbb{R}_+$ with $g(x)\leq 0$ iff $x\leq x^*$. This procedure is widely used in descent direction optimization algorithms with Armijo-type conditions. It both returns an estimate in $(\beta x^*,x^*]$ and enjoys an upper-bound $\lceil \log_{\beta} \epsilon/x_0 \rceil$ on the number of function evaluations to terminate, with $\epsilon$ a lower bound on $x^*$. The basic bracketing mechanism employed in several root-searching methods is adapted here for the purpose of performing inexact line searches, leading to a new class of inexact line search procedures. The traditional bisection algorithm for root-searching is transposed into a very simple method that completes the same inexact line search in at most $\lceil \log_2 \log_{\beta} \epsilon/x_0 \rceil$ function evaluations. A recent bracketing algorithm for root-searching which presents both minmax function evaluation cost (as the bisection algorithm) and superlinear convergence is also transposed, asymptotically requiring $\sim \log \log \log \epsilon/x_0 $ function evaluations for sufficiently smooth functions. Other bracketing algorithms for root-searching can be adapted in the same way. Numerical experiments suggest time savings of 50\% to 80\% in each call to the inexact search procedure. 
\end{abstract}



%
% Keywords. The author(s) should pick words that accurately describe the work being
% presented. Separate the keywords with commas.

\maketitle

\section{Introduction}
Backtracking is an inexact line search technique typically used in the context of descent direction algorithms for solving non-linear optimization problems \citep{luenberger, boyd, gill}. After a descent direction is computed, a step size must be chosen by solving an inexact line searching problem that can be written as
\begin{equation}\label{eq:problem}
    \text{Find } \hat{x} \in \mathbb{R}_+ \text{ such that } g(\hat{x}) \leq 0;
\end{equation}
for some $g:\mathbb{R}_+ \to \mathbb{R}$ such that $g(x)\leq 0$ for all $x$ less than or equal to an unknown turning point $x^*\in \mathbb{R}_+$ and $g(x)>0$ otherwise. The condition $g(x)\leq 0$ expresses some acceptable criteria for a descent method to attain desired convergence properties, such as the well known Armijo's condition \citep{armijo}, Wolfe's condition \citep{wolfe}, amongst others \citep{burachik,shi,boukis,calatroni,truong,vaswani}. The backtracking procedure, initiated with some pre-specified values of $x_0\geq x^*$ and $\beta \in (0,1)$, sequentially verifies and returns $\hat{x}$ as the first value of the sequence $x_0,\ x_0\beta,\ x_0\beta^2, ...$ that satisfies the inequality in (\ref{eq:problem}), i.e. it usually takes no more than three lines (within a larger routine) as described in Algorithm \ref{alg:back_track}.\\ \\
\begin{algorithm}[H]
%\DontPrintSemicolon
  $\tilde{x} \leftarrow x_0$\;
  \While{$g(\tilde{x})>0$}
  {
        $\tilde{x} \leftarrow \beta\tilde{x}$\;
  } 
\caption{Backtracking\label{alg:back_track}}
\end{algorithm}	\emph{ }\\[-3mm]

Notwithstanding the relevance of Algorithm \ref{alg:back_track} as a component of a large variety of nonlinear optimization algorithms, the literature has not focused on its study yet. The working principles of the traditional backtracking algorithm are examined here, and a new class of methods for inexact line search with enhanced performance is proposed. 

It is shown that the traditional backtracking delivers a $\lceil \log_{\beta} \epsilon/x_0 \rceil $ upper-bound on the number of function evaluations to terminate, where $\epsilon$ is a lower bound on $x^*$. The simplest method belonging to the class proposed here, which is based on the traditional bisection algorithm for root-searching, completes the same task with at most $\lceil \log_2 \log_{\beta} \epsilon/x_0 \rceil$ function evaluations. The same upper bound is provided by another method based on a recent bracketing algorithm for root-searching \cite{oliveira1}, which requires asymptotically only $\sim \log \log \log \epsilon/x_0 $ function evaluations in the case of sufficiently smooth functions. Other root-searching bracketing algorithms can be adapted similarly for performing inexact line searches efficiently.

Numerical experiments are provided, suggesting 50\% to 80\% of function evaluation savings in each call to the inexact search procedure.

% =====================================================================
\section{Analysis of Traditional Backtracking}

The procedure described in Algorithm \ref{alg:back_track} enjoys the following guarantees:
\begin{theorem}\label{the:back_track}
Assume that $x^*>\epsilon>0$, $\beta \in (0,1)$, $x_0 > x^*$, and $g(x) > 0$ iff $x > x^*$. Then, the backtracking algorithm finds a solution $\hat{x}$ such that $g(\hat{x})\leq 0$ in at most $\lceil\log_\beta \epsilon/x_0\rceil$ iterations, and the solution $\hat{x}$ satisfies $\beta x^* < \hat{x} \leq x^*$.
\end{theorem}

Theorem \ref{the:back_track} is often an unstated and implicit motivation to employ backtracking, since it both guarantees a finite termination in $\lceil\log_\beta \epsilon/x_0\rceil$ iterations\footnote{The exact number of iterations can be more precisely expressed as a function of $x^*$ with the relation $n =\lceil\log_\beta x^*/x_0\rceil$. The solution-independent bound requires $x^*$ to be bounded away from zero, since otherwise, backtracking may require arbitrarily many iterations the closer $x^*$ is to zero.} and gives a guarantee on the location of $\hat{x}$. The more the value of $\beta$ approximates $1.0$ the closer $\hat{x}$ is guaranteed to be to $x^*$, which is the maximum possible step-size within the guarantees associated with $g(x)\leq 0$. The property that $\beta x^*< \hat{x}\leq x^*$ is often key in ensuring that the parent algorithm ``makes the most out of'' each descent direction expensively computed throughout its iterations. Of course, arbitrarily fast procedures could easily be devised that find $g(x)\leq 0$ by taking faster converging sequences to $0$ if this requirement were to be dropped. Hence implicit to applications that make use of backtracking is the requirement that the solution to problem (\ref{eq:problem}) must be ``not too far from $x^*$''.

Of a similar nature to the requirement that $\hat{x}$ is ``not too far from $x^*$'' is the requirement that $x^*$ is ``not too close to zero''. Without this, the algorithm could take arbitrarily long to find $\hat{x}$ the closer $x^*$ is to zero. This second requirement is, again, implicit in the formulation of backtracking procedures and, at times, it is even entailed by the construction of the parent algorithm. For example, assume the stopping criteria of the parent algorithm verifies stagnation in the domain of the objective function $f(\cdot)$. Then, by construction, when the parent algorithm  finds one instance of (\ref{eq:problem}) such that  ``$x^*$ is too close to zero'', it terminates. Thus, with the exception of the very last iteration, every other iteration will satisfy $x^*\geq \epsilon$. 

In practice, any backtracking procedure should include a stopping condition that is activated when the iteration count $i$ becomes greater than an allowed maximum $i_{max}$, in order to guarantee its termination. This is equivalent to the condition $x_0 \beta^i < \epsilon$ for $\epsilon = x_0 \beta^{i_{max}}$. Hence, the assumption that $x^*\geq \epsilon$ for some pre-specified $\epsilon$ seems to be a hypothesis on (\ref{eq:problem}) that applications that make use of backtracking must assume, either explicitly or implicitly. 

Both of these requirements, extracted from Theorem \ref{the:back_track} and found implicitly or explicitly in the literature, are now stated formally for the sake of clarity. We require that: \\[-3mm]
\\
\textit{\textbf{Condition 1.} For some pre-specified $\beta$ in $(0,1)$, the solution $\hat{x}$ to problem (\ref{eq:problem}) must satisfy $\beta x^*<\hat{x}$.}\\
\textit{\textbf{Condition 2.} For some pre-specified $\epsilon >0$, the turning point $x^*$ of problem (\ref{eq:problem}) satisfies $\epsilon<x^*$.}\\

% ==================================================================
\section{Bracketing-based inexact line search}

The following general algorithm is proposed here:\\[-2mm]

\begin{algorithm}[H]
%\DontPrintSemicolon
  $a \leftarrow \epsilon$\; 
	$b \leftarrow x_0$\;
  \While{$a\leq \beta b$}
  {
        chose $\tilde{x}$ in $(a,b)$ and evaluate $g(\tilde{x})$\; 
        update $(a,b)$ according to (\ref{eq:update})\; 
  } 
  return $\hat{x} = a$;  
\caption{Fast-tracking\label{alg:fast_back_track}}
\end{algorithm}
\emph{ }\\
The update rule in line 5 is defined by: 
\begin{equation}
\left\{
\begin{array}{l}
a \leftarrow \tilde{x} \mbox{ if } g(\tilde{x})<0 \\[1mm]
b \leftarrow \tilde{x} \mbox{ if } g(\tilde{x})>0 \\[1mm]
a \leftarrow \tilde{x} \mbox{ and } b \leftarrow \tilde{x} \mbox{ if } g(\tilde{x}) = 0
\end{array}
\right.
\label{eq:update}
\end{equation}
Algorithm \ref{alg:fast_back_track} defines a class of bracketing-based methods for inexact line search because both the turning point $x^*$ and the final solution $\hat{x}$ are kept inside the interval $[a,b]$ throughout the iterations. Different instances of this algorithm are defined by different choices of $\tilde{x}$ in line 4.

% ---------------------------------------------------------------
\subsection{Geometric bisection fast tracking}
\noindent
Consider the instance of Algorithm \ref{alg:fast_back_track} with the choice of $\tilde{x}$ in line 4 performed according to the choice rule (\ref{eq:choice}):
\begin{equation}
\tilde{x} \equiv \sqrt{ab} 
\label{eq:choice}
\end{equation}
This procedure enjoys the following guarantees:

\begin{theorem}\label{the:fast_back_tracking_worst}
Fast-tracking with $\tilde{x}$ given by (\ref{eq:choice}) finds a solution $\hat{x}$ such that $g(\hat{x})\leq 0$ in at most $ \lceil \log_2 \log_{\beta} \epsilon/x_0 \rceil$ iterations and the solution $\hat{x}$ satisfies $\beta x^* < \hat{x}\leq x^*$.
\end{theorem}
\begin{proof}
The proof follows from the fact that the inequalities $\beta x^* < \hat{x}\leq x^*$ are equivalent to $\log_2 \beta < \log_2 \hat{x} - \log_2x^* \leq 0$, which in turn implies that $|\log_2 \beta |> |\log_2 \hat{x} - \log_2x^*|$. Therefore, to produce an estimate $\hat{x}$ to $x^*$ with \emph{relative} precision of at least $\beta$,  is equivalent to searching for an estimate $\hat{X} = \log_2 \hat{x}$ of $X^* = \log_2 x^*$ with an \emph{absolute} error of at most $-\log_2 \beta$. Under this logarithmic scale, the bisection method is guaranteed to perform the search task with minmax optimality guarantees. What remains is, quite simply, to translate the bisection method from the logarithmic to the standard scale. This is done as follows: Define $A = \log_2 a$ and $B = \log_2b$; thus, if the bisection method takes the midpoint $X_{1/2}=(A+B)/2$ in each iteration on the logarithmic scale, then, in the standard scale  this translates to  $X_{1/2} = (\log_2 a+\log_2 b)/2 = (\log_2 ab)/2 = \log_2 (ab)^{1/2}$. Thus, we have that $\tilde{x}$ must be taken to be equal to $\sqrt{ab}$ in the standard scale.

We now verify that when $B-A\leq -\log_2 \beta$, the lower estimate produced by $A = \log_2 a$ satisfies Condition 1, i.e. that for any value of $x^*$ in $(a,b)$ we must have that $\beta x^*<a$. For this, notice that $B-A\leq -\log_2 \beta \implies \log_2b/a \leq \log_2 \beta^{-1} \implies b/a\leq \beta^{-1}$ which in turn implies that $\beta b\leq a$.  And, since $x^*$ is less than $b$ the inequality in Condition 1 holds. In fact, we express the condition $B-A\leq -\log_2 \beta$ as $a\leq \beta b$ in the standard scale. What is left now is to verify the number of iterations required by the bisection method over the logarithmic scale.

The bisection method requires at most $n_{1/2}\equiv \lceil\log_2 (B_0-A_0)/\delta \rceil$ iterations to reduce the interval $(A,B)$ to one of length $B-A\leq \delta$. Thus, given that $A_0 = \log_2\epsilon$ and that $B_0 = \log_2x_0$  and $\delta = -\log_2 \beta$ we find that $n_{1/2}$ is equal to $\lceil \left(\log_2 (B_0-A_0)/\log_2 \beta\right) \rceil = \lceil\log_2\left(( \log_2 x_0-\log_2\epsilon)/\log_2 \beta\right) \rceil = \lceil\log_2 \log_{\beta} x_0/\epsilon \rceil$.
\end{proof}
Thus, an immediate consequence of Theorem \ref{the:fast_back_tracking_worst} is that naive backtracking procedures unnecessarily fall short in terms of worst case performance. They require exponentially more iterations on the worst case when compared to simple binary searching applied to the logarithmic scale. Of course, the $\lceil \log_\beta \epsilon/x_0 \rceil$ upper-bound of standard backtracking can, and often is, carefully minimized by choosing $x_0$ as ``near as possible'' to $x^*$ by means of interpolation bounds. However, the same procedures that minimize $\lceil \log_\beta \epsilon/x_0 \rceil$ can also be used to minimize the tighter $\lceil \log \log_\beta \epsilon/x_0 \rceil$ upper-bound of {\em geometric bisection fast-tracking}. Notice that backtracking for an estimate with relative precision $\beta$ is equivalent to grid searching with a fixed step size on the logarithmic scale: the relative inefficiencies of grid searching when compared to binary searching are well documented in the literature \citep{press}. Thus, this improvement is attained with no appeal to additional assumptions on the conditions of Problem (\ref{eq:problem}), nor on the use of additional function or derivative evaluations. It is attained solely at the cost of computing the method itself, which for choice rule (\ref{eq:choice}) is the additional computation of one square-root per iteration.

The application of the bisection method on the logarithmic scale seems to be an often forgotten technique within the different communities that make use of numerical solvers, and it is certainly under-represented in the literature. We surveyed popular numerical analysis and optimization textbooks, including \citet{press,chapra,boyd,luenberger,gill}, and found no reference to this technique, despite the existence of scattered references in computational forums \footnote{Some early external references to ``geometric bisection'' can be found in \url{codeforces.com/blog/entry/49189}, \url{math.stackexchange.com/questions/3877202/bisection-method-with-geometric-mean} and \url{github.com/SimpleArt/solver/wiki/Binary-Search}} and other isolated references to ``geometric bisection'' in the context of eigenvalue computation \cite{ralha1,ralha2}. In fact, it is easy to find textbook examples that recommend the use of relative error stopping criteria in conjunction with bisection method on a linear scale (see pseudo-code in Figure 5.11 of \citep{chapra} and chapter 9.1 of \citep{press}).  This gives rise to the same inefficiency as the one caused by the use of naive backtracking. Similar remarks can be made concerning the use of golden-section searching / Fibonacci-searching for a minimum using relative error stopping criteria. Of course, the underlying metric behind floating point arithmetic most certainly prioritizes relative over absolute error in numerical representations \citep{press,chapra}, hence it is natural to recommend upper-bounding relative errors and, for the same reasons, the proper logarithmic scaling should be recommended before the use of bisection type methods, specifically when the initial interval $(a,b)$ can span several orders of magnitude. 

A noticeable exception to the ``inefficiency gap'' between  the use of arithmetic and geometric averages in the bisection method, is when the search is already initiated  with a small interval $(a,b)$ with\footnote{This way, if the standard bisection method runs till $\beta b \leq a$ for some $\beta$ near one, it will take a number of iterations $n_{1/2} $ of the order of $\sim \log_2 (b_0-a_0)/[a_0(1-\beta)] =  \log_2 (b_0-a_0)/a_0 - \log_2 (1-\beta)$, and since $\Delta / x \approx \log_2 (x+\Delta ) - \log x$, we find that $n_{1/2}$ is of the order of $\approx \log_2 (\log_2 b_0 -\log_2 a_0) + \log_2(1-\beta)/\beta = \log_2 \log_2 b_0/ a_0 + \log_2\log_2 \beta$
which simplifies to $\approx  \log_2 \log_\beta b_0/ a_0$, the complexity of the bisection method applied to the logarithmic scale.} $a,b>0$ and with $b-a\ll a$, and thus a value of $\beta$ close to $1$. However, standard conditions under which backtracking  is used can hardly be expected to satisfy this condition since the further into the run of a descent direction algorithm, the closer $x^*$ is expected to be to zero, and hence, quite the opposite is expected. That is, we find that throughout the run of a standard descent direction algorithm $b-a = x_0-\epsilon$ tends to be much greater than $a = \epsilon$. Furthermore, the choice of $\beta$ near one defeats the purpose of employing inexact searching, since it is often intended as a reduction to the computational cost of exact searching. Instead of choosing $\beta$ near one, in this case one might as well employ exact one dimensional minimization techniques to dictate the step-size.

% -------------------------------------------------------
\subsection{Fast tracking with multi-logarithmic speed-up}

Asymptotic bounds are also improved when the proper scale is adopted. This is shown in the following by making explicit the estimated number of iterations when a hybrid technique for the construction of $\tilde{x}$ is used. The exact construction of $\tilde{x}$ is a straightforward application of the ITP root-searching method, described in \citep{oliveira1}, on the logarithmic scale, and is omitted for brevity.

\begin{corollary}\label{cor:ITP}
If $\tilde{x}$ in line 4 of Algorithm \ref{alg:fast_back_track} is taken as the ITP estimate on the logarithmic scale (instead of the bisection method), then, the same guarantees as Theorem \ref{the:fast_back_tracking_worst} hold; and, if furthermore $g(x)$ is $C^1$ with $x^*$ a simple root, then asymptotically the number of iterations is of the order of $\sim \log \log \log_{\beta} \epsilon/x_0$. 
\end{corollary}
\begin{proof}
Follows immediately from the properties of the ITP method \citep{oliveira1}.
\end{proof}

Corollary \ref{cor:ITP} makes use of standard assumptions on the smoothness of $g$, under which even faster convergence can be guaranteed. The ITP method mentioned therein is an efficient first order root-searching method  that in the likes of Ridders' rule, Brent's method or Dekker's method, attains a superlinear order of convergence when employed to solve one dimensional root searching problems. However, unlike the aforementioned methods it is the only one known to retain the minmax optimal performance of the bisection method. The exact inner-workings of the ITP method are beyond the scope of this paper. A reader more familiar with other hybrid methods (such as Ridders', Brent's or Dekker's method) may substitute the ITP method for the solver of preference, albeit with weaker worst case guarantees. The point being that \emph{multi-logarithmic speed-ups can be attained} with interpolation based strategies while retaining the logarithmic speed-up on the worst case performance.

% =================================================================
\section{Experiments}

Quick numerical comparisons between standard backtracking and fast-tracking are performed here under the optimization set-up in which inexact searching is typically employed. For this we implement a standard gradient descent algorithm with Armijo's condition, from which the corresponding function $g(x)$ is derived, to minimize ten different loss functions $f: \mathbb{R}^{10} \mapsto \mathbb{R}$ described in Table \ref{tab:examples}. Both methods were initiated at $\boldsymbol{x} = [1, 1, ... 1]^T$ with $\beta = 0.8$, $\epsilon = 10^{-10},x_0 = 1$ and were compared after twenty gradient descent iterations. All functions chosen contain at least one local minimum not too far from the initial guess, and thus both implementations produced approximately the same path, hence ensuring the comparison is made on as-similar-as-possible conditions. We report here the results using a fixed upper-bound step value for $x_0$ that does not depend on the size of the gradient, i.e. our standard backtracking sequentially searches for the first term in the sequence  $\{ \boldsymbol{x}+\beta^k \nabla f(\boldsymbol{x})/ \|\nabla f(\boldsymbol{x})\| \ \text{ for } k=1,2,...\}$, and fast-tracking calls an external root-searching solver on the logarithmic scale. We use the ITP method\footnote{The ITP parameters used were of $\kappa_1 = 0.1; \kappa_2 = 2$; and, a slack parameter of $N_0 = 0.99$ applied on the rescaled root-searching problem made to satisfy $b-a\leq 1$ in order to benefit of the guarantees of \citep{oliveira1}.}, however other non-linear solvers could have been used with slightly weaker guarantees.  Under the conditions here considered the simple ``geometric average'' bisection method would require exactly 7 function evaluations in each iteration if exact arithmetic were used, hence we use this number as a reference point to which standard backtracking and fast-tracking are compared. 

Figure \ref{fig:evolution} focuses on the first function considered, and displays the evolution of the number of iterations required  by each inexact searching procedure as a function of the gradient-descent iteration.   And, as can be seen, fast-tracking tends to reduce the number of iterations the further into the run while backtracking increases the number of iterations the further into the run. This is because interpolation guarantees of the ITP method are improved with the progression of the gradient run (since it is initialized closer to the final solution), while standard backtracking will require more iterations as the ratio of $x^*/x_0$ is reduced the further into the run. In fact, we observe this pattern of progression of both backtracking and of fast-tracking in most runs.

\begin{figure}
 \centering
 \includegraphics[width=12cm]{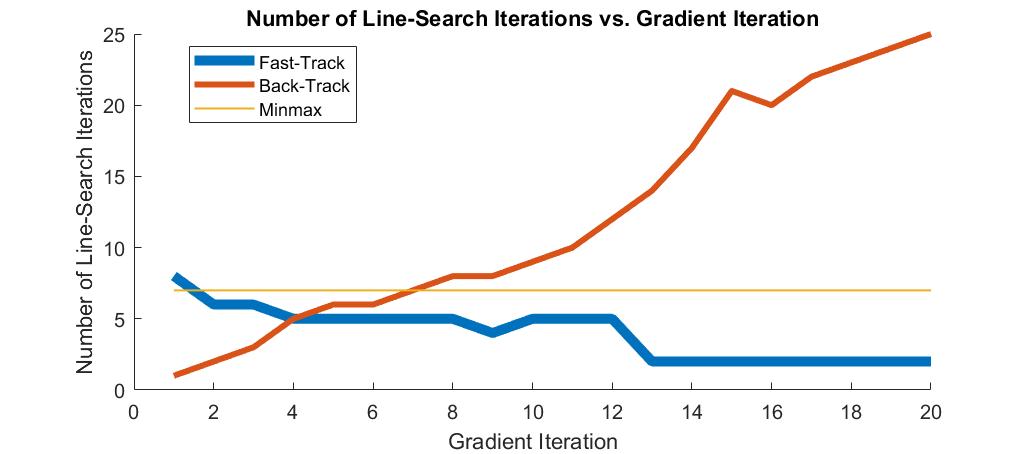} 
 \caption{Evolution of number of iterations after each gradient calclation
\label{fig:evolution}}
\end{figure}

\begin{table}[ht]
\caption{\label{tab:examples} Average number of function evaluations required to solve  the inexact line-search problem in each iteration of a vanilla gradient descent for different loss functions. The numbers reported are the average obtained after 20 gradient steps under conditions where the minmax ``geometric-average'' binary-searching procedure would require exactly 7 iterations. Bellow, the symbol $V$ stands for an identity matrix plus the Vandermonde matrix  obtained in interpolation problems on $n$ Chebyshev points; the vector $\boldsymbol{n}$ is defined as $[1, 2, ..., n]^T$, and every operation on $\boldsymbol{n}$ is done element-wise. }
\begin{tabular}{lc|cc|} & & \rotatebox[origin=c]{-90}{Backtracking \hspace{7pt}} & \rotatebox[origin=c]{-90}{Fast-tracking \hspace{8pt}}     \\
\multicolumn{2}{l}{\textbf{Functions -}} \textbf{ }    &\hspace{-22.3pt}\raisebox{4pt}{|} \hspace{-4.55pt}\raisebox{-4pt}{|} \textbf{ }  & \textbf{ }  \\ 
Simple Quadratic &   $\sum_i x_i^2$ & 12.2 & 4.0 \\
High Degree Polynomial & $\sum_i x_i^{2i}$ & 10.8 & 4.8 \\
Vandermonde Interpolation & $\boldsymbol{x}^TV\boldsymbol{x}$ & 16.2 & 3.9 \\
Trigonometric 1 & $\sum_i i\cos(x_i)$  &  7.0 &  4.8 \\
Trigonometric 2 &  $\sum_i i\cos(\cos(x_i))$ & 12.0 & 4.0 \\
Log-Poly & $2\log ||\boldsymbol{x}-\boldsymbol{n}^{1/\boldsymbol{n}}||_2$ & 27.2 & 2.7 \\
Quartic & $\tfrac{1}{n}(\sum_i x_i)^4+|\sqrt{\boldsymbol{n}^Tx}|$ & 24.7 & 4.2 \\
Interpolation w/ Regularizer & $\boldsymbol{x}^TV\boldsymbol{x}+||\boldsymbol{x}-\sqrt{\boldsymbol{n}}||_1$ & 49.5 & 3.5 \\
Noisy Quadratic Hard & $||\boldsymbol{x}||^2_2+10^{-3}\sum_i \sin(i/x_i)$ & 26.8 & 3.0 \\
Noisy Quadratic Easy & $||\boldsymbol{x}||^2_2+10^{-3}\sum_i \sin(10^3ix_i)$  & 35.6 & 2.3 \\
& & \textbf{ }  & \textbf{ } \\
\textbf{Global Average} & & \textbf{22.2} & \textbf{3.7}   \\
\textbf{Global Worst Case} &  & \textbf{147} & \textbf{8}  \\  
\end{tabular}
\end{table}

As can be seen in Table \ref{tab:examples}, fast-tracking vastly improves on standard backtracking under both average and worst-case performance. The global average of fast-tracking is roughly 50\% that of the minmax guarantee of 7 iterations, and, since the ITP solver called made use of the $0.99$ slack variable, the worst case performance over the test set is at most $\lceil 0.99\rceil = 1$ iteration more than the minmax guarantee, i.e. at most $7+1 = 8$ iterations. Standard backtracking only attained a number of iterations equal to the minmax on one instance, and was outperformed by vanilla ``geometric average'' binary searching on every other instance.

Furthermore, by varying the values of $\beta$ and the initial estimate $\boldsymbol{x}$, we verify that the differences in performance are affected too. Our preliminary estimates suggest that for values of $\beta$ near $0.5$ backtracking performs much worse than fast-tracking than what is reported in Table \ref{tab:examples}, multiplying by a factor of $10$ the difference in average iteration count at it's peak value. For $\beta$ near $0$ or $1$ the differences are kept roughly in the range of the ones reported in Table \ref{tab:examples}. Concerning the effect of the initial estimate for $\boldsymbol{x}$, our experiments suggest that the the closer the initial estimate is to the stationary point $\boldsymbol{x}^*$ to which the gradient method converges, the greater the benefit of fast-tracking over backtracking, and when initiated far from $\boldsymbol{x}^*$ the difference in performance is reduced, but not reversed. Finally, analogous experiments were also performed providing the solvers with additional interpolation information and different values of $x_0$ and found no significant difference in the comparative performance reported above. Thus, these results have been kept out for brevity.

\section{Discussion}
The emphasis of ``backtracking papers'' does not typically lie on the three lines that construct and verify which point in the sequence $x_0, x_0\beta ...$ first satisfies $g(x)\leq 0$. In fact, the construction of $g(\cdot)$, and the guarantees associated with the criteria $g(\cdot)\leq 0$, is typically where the contributions of those papers are found. Thus, perhaps justifiably so, it seems that not much research effort has been devoted to those three lines since they, informally speaking, ``get the job done'' and ``some other paper can deal with it''. This is that paper.

Here, we show a simple and proper construction of a procedure that finds $g(x)\leq 0$, and does so with optimal guarantees. A logarithmic speed-up is attained with respect to worst case, and a multi-logarithmic speed-up is attained with respect to asymptotic performance if hybrid interpolation based techniques are employed. These speed-ups are well reflected in experiments achieving roughly 50\% to 80\% time savings in each call to the inexact line-searching procedure.

% -------------------------------------

\begin{acks}
This paper was written when the first author was a graduate student at the Federal University of Minas Gerais.
\end{acks}

%
% The next two lines define the bibliography style to be used, and the bibliography file.
\bibliographystyle{ACM-Reference-Format}
\bibliography{sample-base}

%
% If your work has an appendix, this is the place to put it.
%\appendix


\end{document}